\documentclass[10pt]{amsart}
\pdfoutput=1
\usepackage{mathtools}
\usepackage[mathscr]{eucal}
\usepackage{ifpdf}
\usepackage{url}
\usepackage{amsmath}
\usepackage{graphicx} \usepackage[all]{xy}
\usepackage{caption, subcaption}
\usepackage[usenames,dvipsnames]{color}
\usepackage{graphicx}
\usepackage[text={5.58in,8.5in},centering,letterpaper,dvips]{geometry}
\usepackage{amsfonts}
\usepackage{epsf}
\usepackage{amssymb}
\usepackage{amsmath}
\usepackage{amscd}
\usepackage{pdfpages}
\usepackage{fancyhdr}
\usepackage{setspace}
\usepackage[all]{xy}
\usepackage{verbatim}
\usepackage{enumerate}
\usepackage{pgf}
\usepackage{comment}
\usepackage[normalem]{ulem}
\usepackage{tikz}
\usetikzlibrary{
  cd,
  shapes,
  calc,
  positioning,
  arrows,
  decorations.pathreplacing,
  decorations.markings,
  intersections,
  patterns,
  backgrounds
}
\usepackage[latin1]{inputenc}
\usepackage[colorlinks=true, urlcolor=NavyBlue, linkcolor=NavyBlue, citecolor=NavyBlue]{hyperref}

\theoremstyle{theorem}
\newtheorem{theorem}{Theorem}[section]

\newtheorem{lemma}[theorem]{Lemma}
\newtheorem{corollary}[theorem]{Corollary}

\makeatletter
\newtheorem*{rep@theorem}{\rep@title}
\newcommand{\newreptheorem}[2]{%
\newenvironment{rep#1}[1]{%
 \def\rep@title{#2 \ref{##1}}%
 \begin{rep@theorem}}%
 {\end{rep@theorem}}}
\makeatother

\newreptheorem{theorem}{Theorem}
\newreptheorem{lemma}{Lemma}
\newreptheorem{claim}{Claim}
\newreptheorem{question}{Question}
\newreptheorem{corollary}{Corollary}
\newreptheorem{proposition}{Proposition}

\theoremstyle{definition}
\newtheorem{definition}[theorem]{Definition}

\newtheorem{remark}[theorem]{Remark}

\usepackage{accents}
\newlength{\dhatheight}

\author{Ryan Blair, Alexandra Kjuchukova and Ella Pfaff}

\begin{document}

\rhead{\thepage}
\lhead{\author}
\thispagestyle{empty}


\raggedbottom
\pagenumbering{arabic}
\setcounter{section}{0}


\title{The Plain Sphere Number of a Link}
\thanks{MSC codes: 57K10, 57K31, 57M05.}


\maketitle


\begin{abstract}
Let $L$ be a link in $S^3$. We consider the class of meridional presentations for $\pi_1(S^3\backslash L)$ in which the relations are witnessed by  embedded two-spheres which can be represented simultaneously in a fixed diagram of $L$, analogously to decomposition spheres studied by Cromwell, Menasco and others. Wirtinger relations are witnessed by such spheres and the Wirtinger presentation is a special case of the ones we study. We prove that the smallest number of generators of $\pi_1(S^3\backslash L)$ over all such presentations equals the bridge number of $L$.
	\end{abstract}

\section{Introduction}
We introduce a new definition of bridge number of a link. The definition arises in the study of the Meridional Rank Conjecture (MRC), which asks whether the bridge number of a link $L$ equals the smallest number of meridional generators of $\pi_1(S^3\backslash L)$.
The conjecture, posed by Cappell and Shaneson~\cite[Problem 1.18]{kirby1995problems}, has been established in a variety of cases~\cite{boileau1985nombre, rost1987meridional, burde1988links, boileau1989orbifold, boileau2017meridionalrank, CH14, Corn14, baader2019coxeter, baader2017symmetric, baader2023bridge, Dutra22}. The analogous statement is also shown to hold for some knotted spheres in $S^4$~\cite{joseph2024meridional}. 

In~\cite{blair2020wirtinger}, it is shown that the bridge number equals the minimal number of meridional generators over all presentations which allow only iterated Wirtinger relations in a fixed diagram. Our main result is that equality persists after significantly generalizing the presentations considered. 

Denote the \emph{bridge number} and \emph{meridional rank} of a link $L$ by $\beta(L)$ and $\mu(L)$, respectively. 
In \cite{blair2020wirtinger}, the authors introduce a new invariant, the \emph{Wirtinger number} of a link, denoted $\omega(L)$. The Wirtinger number is defined in terms of a coloring procedure carried out in a fixed diagram of $L$, see Definition~\ref{def:Wirtcolor}. Performing a coloring move at a crossing $c$ reflects the fact that the Wiritnger meridians of the overstrand and one understrand at $c$ generate the Wirtinger meridian of the second understrand at $c$. Hence, a valid coloring sequence for $D$ demonstrates that the Wirtinger meridians of the initially colored strands, or seeds (Definition~\ref{def:seeds}), generate $\pi_1(S^3\backslash L)$.

By starting with a diagram in minimal bridge position and choosing the strands containing the local maxima as seeds, we easily see that $\beta(L) \geq \omega(L)$; moreover, by definition, $\omega(L) \geq \mu(L)$. In \cite{blair2020wirtinger}, it is shown that in fact $\beta(L) = \omega(L)$. We now prove that the bridge number equals the smallest number of meridional generators of $\pi_1(S^3\backslash L)$ across a considerably more general class of meridional presentations. 

The {\it plain sphere number} of a link $L$, denoted $\rho(L)$ (Definition~\ref{def:loopcolor}), is the smallest number of Wirtinger meridians which generate $\pi_1(S^3\backslash L)$ using only relations witnessed by certain embedded two-spheres which we now describe. Consider a link diagram $D$, and let $\gamma$ denote an embedded circle in the plane of projection such that $\gamma$ avoids neighborhoods of crossings and meets the interior of strands of $D$ transversely. Further assume that $\gamma$ meets $D$ in exactly $n$ points and that, of those, precisely one is contained in a given strand, $s$, of $D$. Then, the Wirtinger meridians of the remaining strands which meet $\gamma$ generate the Wirtinger meridian of $s$. To see this, cap off the simple closed curve $\gamma$ with two disks, $D^2_\pm$, whose interiors are disjoint from the plane of projection and such that $D^2_+$ (resp. $D^2_-$) is above (resp. below) the plane. We refer to $D^2_+\cup_\gamma D^2_-$ as a {\it plain sphere}---it is indeed in plain sight---and we say that the sphere witnesses a relation in the group. That is, the product of (appropriately oriented) Wirtinger meridians of strands intersecting $\gamma$, taken in the order determined by $\gamma$, is trivial. The Wirtinger relation at any crossing $c$ corresponds to a circle $\gamma$ equal to the boundary of a small planar neighborhood of $c$. 
Clearly, $\beta(L) = \omega(L)\geq \rho(L) \geq \mu(L)$. 

\begin{theorem}\label{thm:main}
    Let $L$ be a link in $S^3$. The plain sphere number of $L$ equals the bridge number of $L$.
\end{theorem}

The theorem is proved in Section~\ref{sec:Proof}. In Section~\ref{sec:Preliminaries} we give the formal definition of $\rho(L)$, recall the definition of $\omega(L)$, and establish some terminology. The short Section~\ref{sec:Examples} provides an example contrasting the plain sphere and Wirtinger numbers of a diagram, and describes a procedure for computing $\rho(D)$ for a fixed diagram. In Section~\ref{sec:poem} we draw additional figures.

\section{Preliminaries} \label{sec:Preliminaries}

 Recall that if $L$ is a link in $\mathbb{R}^3$ and $P:\mathbb{R}^3\rightarrow \mathbb{R}^2$ is the standard projection map given by $P(x,y,z)=(x,y)$, then $P(L)$ is a \textit{link projection} if $P|_L$ is a regular projection. Hence, a link projection is a finite four-valent graph in the plane, and we refer to the vertices of this graph as crossings. A \textit{link diagram} is a link projection together with labels at each crossing that indicate which strand goes over and which goes under. By standard convention, these labels take the form of deleting parts of the under-arc at every crossing, and thus we think of a link diagram as a disjoint union of closed arcs, or \emph{strands}, in the plane, together with instructions for how to connect these strands to form a union of simple closed curves in $\mathbb{R}^3$. At times we will refer to the link projection $P(D)$ corresponding to a given link diagram $D$, where $P(D)$ is obtained from $D$ by forgetting the under and over information at crossings.

Let $D$ be a diagram of a link $L$ with $n$ crossings.
Denote by $s(D)$ the set of strands $s_1$, $s_2$, \dots , $s_n$ and let $v(D)$ denote the set of crossings $c_1$, $c_2$, \dots, $c_n$. Two strands $s_i$ and $s_j$ of $D$ are {\it adjacent} if $s_i$ and $s_j$ are the understrands of some crossing in $D$. The diagram of the unknot with a single crossing is the unique knot diagram up to planar isotopy for which there exists a strand $s_i$ of $D$ that is adjacent to itself. In the case of links, many diagrams contain a strand adjacent to itself; the standard diagram of the Hopf link is one example.

\begin{definition}\label{def:Wirtcolor}
    We call $D$ \textit{Wirtinger $k$-colorable} if we have specified a set $A$ of strands of $D$ and a nested sequence of subsets $A=A_0\subset A_1 \subset \dots \subset A_{|s(D)|-|A|}=s(D)$ such that the following hold:
\begin{enumerate}
   \item $|A|=k$;
    \item $A_{i+1} \setminus A_{i} = \{s_j\}$ for some strand $s_j$ in $D$;
    \item Whenever $A_{i+1} \setminus A_{i} = \{s_j\}$, $s_j$ is adjacent to some $s_i \in A_i$ at a crossing $c \in v(D)$;
    \item The overstrand $s_k$ at $c$ is an element of $A_i$.
\end{enumerate}
\end{definition}

When $D$ is $k$-colorable in the above sense, the Wirtinger meridians of the strands in $A$ generate all Wirtinger meridians in the diagram, using only Wirtinger relations at crossings.

\begin{definition}\label{def:loopcolor}
    We call $D$ \textit{plain sphere $k$-colorable} if we have specified a set $A$ of strands of $D$ and a nested sequence of subsets $A=A_0\subset A_1 \subset \dots \subset A_{|s(D)|-|A|}=s(D)$ such that the following hold:
\begin{enumerate}
    \item $|A|=k$;
    \item $A_{i+1} \setminus A_i = \{s_j\}$ for some strand $s_j$ in $D$;
    \item \label{item:loop-move} Whenever $A_{i+1} \setminus A_i = \{s_j\}$, there exists an embedded circle $L_{i+1}$ in the plane of projection such that: $L_{i+1}$ is transverse to the projection $P(D)$; $L_{i+1}$ is disjoint from small neighborhoods of crossings in $D$; $|L_{i+1}\cap s_j|=1$; and the points of intersection between $L_{i+1}$ and strands of $D$ are all contained in $A_{i+1}$.
\end{enumerate}
We refer to the circles $L_i$ as {\it loops}. 
\end{definition}
Condition~\ref{item:loop-move} precisely guarantees that the Wirtinger meridian of $s_j$ is generated by the Wirtinger meridians of the remaining strands of $D$ that have non-trivial intersection with $L_{i+1}$. This is expressed by a relation in $\pi_1(S^3\backslash L)$ and is witnessed by a plain sphere as defined in the introduction: an embedded two-sphere $S^2_j$ which intersects the plane of projection in the loop $L_{i+1}$. The existence of a valid coloring sequence as above shows that the Wirtinger meridians of the strands in $A$ generate all Wirtinger meridians in the diagram, using only relations witnessed by plain two-spheres. 
\begin{remark}
We could almost regard the spheres $S^2_j$ in the above description as simultaneously and disjointly embedded in $S^3$. Indeed we will show later that the loops $L_i$ can be assumed pairwise non-intersecting. And, each $S^2_j$ is the boundary union of two embedded disks, on opposite sides of the plane of projection, whose interiors do not meet. However, since the upper disk $D^2_{j+}$ of each $S^2_j$ contains the basepoint, the spheres do, in fact, intersect at a point.     
\end{remark}

\begin{definition}\label{def:seeds}
    When a diagram $D$ can be colored by a valid sequence of Wiritinger coloring moves (resp. plain sphere coloring moves) as in Definition~\ref{def:Wirtcolor} (resp. Definition~\ref{def:loopcolor}), the elements of $A$ are called the {\it seed strands}, or simply {\it seeds}, for the coloring.
\end{definition}

The minimum value of $k$ such that $D$ is Wirtinger $k$-colorable is the {\it Wirtinger number} of $D$, denoted $\omega(D)$. Similarly, the minimum value of $k$ such that $D$ is plain sphere $k$-colorable is the {\it plain sphere number} of $D$, denoted $\rho(D)$. We use $\omega(D)$ and $\rho(D)$ to define invariants of $L$. 

\begin{definition}
    \label{omega}  Let $L\subset S^3$ be a link. The {\it Wirtinger number} of $L$, denoted $\omega(L)$, is the minimal value of $\omega(D)$ over all diagrams $D$ of $L$. Similarly, the {\it plain sphere number} of $L$, denoted $\rho(L)$, is the minimal value of $\rho(D)$ over all diagrams $D$ of $L$.
\end{definition}

We now define a number of auxiliary terms related to Definitions \ref{def:Wirtcolor}, \ref{def:loopcolor} and \ref{omega}. These terms will be used extensively in the sections that follow. If $D$ is a Wirtinger $k$-colorable diagram with $A_{i} \setminus A_{i-1} = \{s_j\}$, we say that the strand $s_j$ is colored at \emph{stage} $i$ by a \emph{Wirtinger coloring move}. Similarly, if $D$ is a plain sphere $k$-colorable diagram with $A_i \setminus A_{i-1} = \{s_j\}$, we say that the strand $s_j$ is colored at \emph{stage} $i$ by a \emph{plain sphere coloring move} or, for short, a {\it loop coloring move}. A \emph{partially plain sphere colored} link diagram $D$ is a nested collection of sets  $A=A_0\subset A_1 \subset \dots \subset A_{r}$ that meet all of the requirements of Definition \ref{def:loopcolor} with the exception that $A_r$ may be a proper subset of $s(D)$. Additionally, if $D$ is a plain sphere $k$-colorable diagram, then a \emph{plain sphere coloring sequence} for $D$ is an ordered set of loops $\mathcal{L}=(L_1, L_2, \dots, L_{|s(D)|-k})$ that are as in~(\ref{item:loop-move}) of Definition \ref{def:loopcolor}. Finally, if $L_i$ is a loop in 
 a plain sphere coloring sequence for $D$ that is isotopic (via an isotopy that is transverse to the link projection) to the boundary of a regular neighborhood of a crossing of $D$, then we call $L_i$ a \emph{Wirtinger loop}. Clearly, any plain sphere coloring move performed using a Wirtinger loop is also achievable via a Wirtinger coloring move. In other words, a Wirtinger coloring sequence is merely a special case of a plain sphere coloring sequence. This shows that $\rho(D)\leq \omega(D)$ for all link diagrams $D$, and hence $\rho(L)\leq \omega(L)$ for all links $L$ in $S^3.$

An easy example of a plain sphere coloring move which is not a Wirtinger move can be found in any diagram which is not visually prime. Namely, the connected sum sphere is a plain sphere intersecting two strands whose meridians cobound an annulus in the link complement. Two additional plain sphere coloring moves which are not Wirtinger moves are depicted in Figure  \ref{fig:got-the-moves}. In Figure~9 of \cite{blair2019incompatibility}, the reader can find an example of a minimal diagram $D$ of a composite knot $K$ such that $\rho(D)=\beta(D)=5$ while $\omega(D)$=6. A diagram $D$ of a prime knot with $\rho(D)<\omega(D)$ is given in Figure~\ref{fig:14n1527_colored}. In particular, this exhibits a plain sphere move which is neither a Wirtinger move nor a loop arising from a connected sum sphere. For a family of plain sphere moves see Figure~\ref{fig:got-the-moves}.

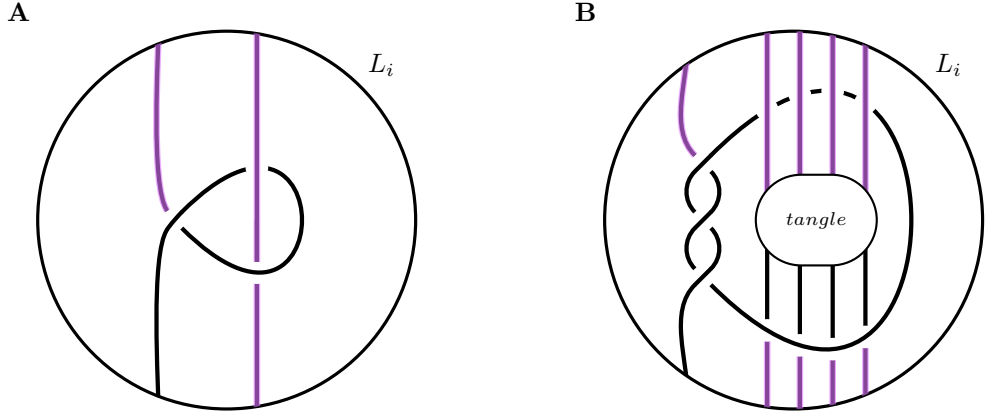
\begin{figure}[h]
    \centering

        \begin{tikzpicture}
        [x=1mm,y=1mm, ultra thick, 
        help lines/.style={very thin, draw=white},
        crossing/.style={circle, draw, white, fill=white, minimum size=5pt, node font=\tiny},
        colored/.style={line width=2.5pt, draw=lightpurple, draw opacity=0.6}]
        \definecolor{lightpurple}{HTML}{DF73FF}

        \draw[white] (-27.5,0) -- (102.5,0);
        \draw[very thick] (0,0) circle (25) node at (45:29) {$L_i$};
        \draw[very thick, xshift=75mm] (0,0) circle (25) node at (45:29) {$L_i$};
        \node at (-27.5,27.5) {\textbf{A}};
        \node[xshift=75mm] at (-27.5,27.5) {\textbf{B}};

        \begin{scope}[xshift=75mm]
            \clip (0,0) circle (25);
            
            \draw[name path=h1, help lines] (-3.5,-25) -- (-3.5,25);
            \draw[name path=h2, help lines] (0.83,-25) -- (0.83,25);
            \draw[name path=h3, help lines] (5.17,-25) -- (5.17,25);
            \draw[name path=h4, help lines] (9.5,-25) -- (9.5,25);

            \node[crossing] (1) at (-12,7.5) {};
            \node[crossing] (3) at (-12,-7.5) {}
                edge[name path=tcurve, help lines, bend right=135, distance=50mm] (1);
            
            \path[name intersections={of=tcurve and h1,by={b1,t1}}];
            \path[name intersections={of=tcurve and h2,by={b2,t2}}];
            \path[name intersections={of=tcurve and h3,by={b3,t3}}];
            \path[name intersections={of=tcurve and h4,by={b4,t4}}];

            \draw (-3.5,-25) -- (-3.5,-3)
            (0.83,-25) -- (0.83,-3)
            (5.17,-25) -- (5.17,-3)
            (9.5,-25) -- (9.5,-3);
            \draw[colored] (-3.5,-25) -- (b1)
            (0.83,-25) -- (b2)
            (5.17,-25) -- (b3)
            (9.5,-25) -- (b4);
            \node[crossing] at (b1) {};
            \node[crossing] at (b2) {};
            \node[crossing] at (b3) {};
            \node[crossing] at (b4) {};

            \node[crossing] (2) at (-12,0) {};
            \node[crossing] (1) at (-12,7.5) {}
                edge[out=315, in=45] (2)
                edge[bend right=45] (2);
            \node[crossing] (3) at (-12,-7.5) {}
                edge[bend right=45] (2)
                edge[bend left=45] (2)
                edge[bend right=135, distance=50mm] (1);
            \node[crossing] (4) at (-14,25) {}
                edge[out=270, in=135, distance=5mm] (1)
                edge[colored, out=270, in=135, distance=5mm] (1);
            \node[crossing] (5) at (-14,-25) {}
                edge[out=90, in=225, distance=5mm] (3);
            \draw[shift=(1)] (225:4.5pt) -- (45:4.5pt);
            \draw[shift=(2)] (225:4.5pt) -- (45:4.5pt);
            \draw[shift=(3)] (225:4.5pt) -- (45:4.5pt);
            
            \node[crossing] at (t1) {};
            \node[crossing] at (t2) {};
            \node[crossing] at (t3) {};
            \node[crossing] at (t4) {};
            \draw (-3.5,3) -- (-3.5,25)
            (0.83,3) -- (0.83,25)
            (5.17,3) -- (5.17,25)
            (9.5,3) -- (9.5,25);
            \draw[colored] (-3.5,3) -- (-3.5,25)
            (0.83,3) -- (0.83,25)
            (5.17,3) -- (5.17,25)
            (9.5,3) -- (9.5,25);
            \draw[thick, fill=white, rounded corners=6mm] (-5,-6) rectangle (11,6) node[midway] {$\scriptstyle tangle$};
        \end{scope}

        \begin{scope}[xshift=3mm]
            \clip (-3,0) circle (25);
            \draw[name path=h1, help lines] (1,-25) -- (1,25);

            \node[crossing] (1) at (-10,0) {}
                edge[name path=rcurve, help lines, out=315, in=45, distance=30mm] (1);

            \path[name intersections={of=rcurve and h1,by={b1,t1}}];

            \draw (1,-25) -- (1,0);
            
            \draw[colored] (1,-25) -- (b1);
            
            \node[crossing] at (b1) {};
            \filldraw[name path=bottomcircle, very thin, white] (b1) circle (4pt);
            \path[name intersections={of=bottomcircle and h1,by={p1,p2}}];

            \node[crossing] (1) at (-10,0) {}
                edge[out=315, in=45, distance=30mm] (1);
            \node[crossing] (4) at (-12,27.5) {}
                edge[out=270, in=125, distance=4mm] (1)
                edge[colored, out=270, in=125, distance=4mm] (1);
            \node[crossing] (5) at (-12,-27.5) {}
                edge[out=90, in=235, distance=4mm] (1);
            \draw[shift=(1)] (235:4.5pt) .. controls +(60:2pt) and ++(225:2pt) .. (45:4.5pt);

            \node[crossing] at (t1) {};
            
            \draw (p1) -- (1,25);
            
            \draw[colored] (p1) -- (1,25);
            
        \end{scope}
    \end{tikzpicture}

    \caption{Possible plain sphere coloring moves at stage $i$ of a coloring process which do not reduce to sequences of Wirtinger moves in the given tangle diagrams. {\bf A} is the move used in Figure~\ref{fig:14n1527_colored}. {\bf B} is a generalization. The oval contains an arbitrary $n$-strand tangle, pictured in the case $n=4$. The remaining tangle strand can be replaced by any one-strand tangle; the move will remain valid as long as all but one of the points in $L_i\cap D$ are colored before stage $i$.}
    \label{fig:got-the-moves}
\end{figure}

\section{Proof}\label{sec:Proof}

We prove Theorem \ref{thm:main} by showing that the plain sphere number of $L$ equals the Wirtinger number of $L$. The result then follows from~\cite{blair2020wirtinger}, where it is shown that the Wirtinger number and bridge number are equal. The equality $\rho(L)=\omega(L)$ relies on the following lemmas.

\begin{figure}[h]
    \centering

    \begin{tikzpicture}
        [x=1mm,y=1mm, thick, 
        help lines/.style={very thin, draw=white}]
        \definecolor{lightpurple}{HTML}{DF73FF}
        \definecolor{darkpurple}{HTML}{9932CC}
        \definecolor{darkblue}{HTML}{003366}
        \definecolor{pink1}{HTML}{FF1493}
        
        \begin{scope}
        \clip (0,0) circle (33);
        \coordinate (o0) at (0,0);
        \coordinate (o1) at (94,25);
        \coordinate (o2) at (24.5,-31.5);
        \coordinate (o3) at (15,-84);
        \coordinate (o4) at (6,9);
        \coordinate (o5) at (-13,-6);

        \draw[name path=h0, help lines] (0,0) circle (23.75);
        \draw[name path=h1, help lines] (o1) circle (91);
        \draw[name path=h2, help lines] (o2) circle (36);
        \draw[name path=h3, help lines] (o3) circle (71);
        \draw[name path=h4, help lines] (o4) circle (7);
        \path[name intersections={of=h0 and h1,by={star0,star0'}}];
        \path[name intersections={of=h1 and h2,by={star1, star1'}}];
        \path[name intersections={of=h2 and h3,by={star2',star2}}];
        \path[name intersections={of=h3 and h0,by={star3',star3}}];
        \path[name intersections={of=h1 and h4,by={star4,star4'}}];
        \end{scope}

        \draw[name path=c0] (0,0) circle (25);

        \begin{scope}
        \clip (0,0) circle (28);
        \draw[name path=c1] (o1) circle (90);
        \draw[name path=c2] (o2) circle (35);
        \draw[name path=c3] (o3) circle (70);
        \draw[name path=c4] (o4) circle (6);
        \path[name intersections={of=c0 and c1,by={A,A'}}];
        \path[name intersections={of=c1 and c2,by={B,B'}}];
        \path[name intersections={of=c2 and c3,by={C',C}}];
        \path[name intersections={of=c3 and c0,by={D',D}}];
        \path[name intersections={of=c4 and c1,by={E,E'}}];

        \draw[fill=black!05] 
        let \p1 = ($ (o1) - (A) $),
            \p2 = ($ (o1) - (B) $),
            \p3 = ($ (o2) - (B) $),
            \p4 = ($ (o2) - (C) $),
            \p5 = ($ (o3) - (C) $),
            \p6 = ($ (o3) - (D) $),
            \p7 = (D),
            \p8 = (A),
            \p9 = ($ (o4) - (E) $),
            \p{10} = ($ (o4) - (E') $),
            \p{11} = ($ (o1) - (E) $),
            \p{12} = ($ (o1) - (E') $)
        in  (A) arc (atan2(\y1,\x1)+180:atan2(\y{11},\x{11})+180:90) -- (E) arc (atan2(\y9,\x9)+180:atan2(\y{10},\x{10})+180:6) -- (E') arc (atan2(\y{12},\x{12})+180:atan2(\y2,\x2)+180:90) -- (B) arc (atan2(\y3,\x3)+180:atan2(\y4,\x4)+180:35) --(C) arc (atan2(\y5,\x5)+180:atan2(\y6,\x6)+180:70) -- (D) arc (atan2(\y7,\x7)+360:atan2(\y8,\x8):25) -- cycle;

        \draw[dashed, line join=round] 
        let \p1 = ($ (o1) - (star0) $),
            \p2 = ($ (o1) - (star1) $),
            \p3 = ($ (o2) - (star1) $),
            \p4 = ($ (o2) - (star2) $),
            \p5 = ($ (o3) - (star2) $),
            \p6 = ($ (o3) - (star3) $),
            \p7 = (star3),
            \p8 = (star0),
            \p9 = ($ (o4) - (star4) $),
            \p{10} = ($ (o4) - (star4') $),
            \p{11} = ($ (o1) - (star4) $),
            \p{12} = ($ (o1) - (star4') $)
        in  
        (star0) arc (atan2(\y1,\x1)+180:atan2(\y{11},\x{11})+180:91) -- (star4) arc (atan2(\y9,\x9)+180:atan2(\y{10},\x{10})+180:7) -- (star4') arc (atan2(\y{12},\x{12})+180:atan2(\y2,\x2)+180:91) -- (star1) arc (atan2(\y3,\x3)+180:atan2(\y4,\x4)+180:36) -- (star2) arc (atan2(\y5,\x5)+180:atan2(\y6,\x6)+180:71) -- (star3) arc (atan2(\y7,\x7)+360:atan2(\y8,\x8):23.75) -- cycle;
        \end{scope}

        \draw[name path=c5, fill=white] (o5) circle (4.25);

        \draw[name path=uncolored, ultra thick] (150:21) -- (150:29);

        \draw[ultra thick] 
        let \p9 = (o2),
            \p0 = (o1)
        in
        (97:21) -- (97:29)
        (197:21) -- (197:29)
        (357:21) -- (357:29)
        (61:21) -- (61:29)
        (69:21) -- (69:29)
        {[xshift=\x9,yshift=\y9](139:31) -- (139:39)}
        {[xshift=\x0,yshift=\y0](203:86) -- (203:94)};
        \draw[line width=2.5pt, lightpurple, draw opacity=0.6] 
        let \p9 = (o2),
            \p0 = (o1)
        in
        (97:20.75) -- (97:29.25)
        (197:20.75) -- (197:29.25)
        (357:20.75) -- (357:29.25)
        (61:20.75) -- (61:29.25)
        (69:20.75) -- (69:29.25)
        {[xshift=\x9,yshift=\y9](139:30.75) -- (139:39.25)}
        {[xshift=\x0,yshift=\y0](203:85.75) -- (203:94.25)};

        \draw[line width=2.5pt, draw opacity=0.4] let \p7 = (D), \p8 = (A) 
        in (D) arc (atan2(\y7,\x7)+360:atan2(\y8,\x8):25);
        \path[name intersections={of=c0 and uncolored,by=x}];
        \filldraw (x) circle (2pt);
        \node at ($(x)+(195:3.5)$) {$x$};
        \node at (-46:27.5) {$\mathcal L$};
        \node at (45:29) {$L_i$};
        \node at (125:27.5) {$a$};
        \node at ($(o4)+(145:10.75)$) {\textcolor{black!80}{$L_i^*$}};
        \node at (159:11.25) {\textcolor{black!50}{$E_i^*$}};
      
    \end{tikzpicture}
    
    \caption{A region in the plane containing a link diagram at stage $i$ of the coloring process. Subarcs of the diagram are represented by short line segments, and loops in $\mathcal L$ are represented by circles and arcs. The loop $L_i$, which colors the strand containing the point $x$, can be replaced with the loop $L_i^*$ (dashed line), which is disjoint from all other loops in $\mathcal L$ and represents a valid plain sphere coloring move at stage $i$ for the same strand. 
    }
    \label{fig:intersecting_loops}
\end{figure}
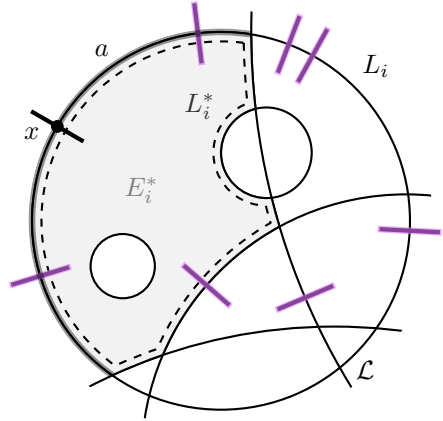

\begin{lemma}\label{lem:disjoint-loops} Let $D$ be a link diagram. The plain sphere number of $D$ can be realized by a collection of disjoint circles. That is, if $\rho(D)=n$, then there exists a set of $n$ seed strands in $D$, together with an ordered set $\mathcal{L}=(L_1, \dots, L_{|s(D)|-n})$ of disjointly embedded circles, each transverse to $D$ and disjoint from small neighborhoods of all crossings, such that $\mathcal{L}$ defines a valid plain sphere coloring sequence for $D$.
\end{lemma}

\begin{proof}
    Let $\mathcal{L}$ be an ordered collection of loops realizing the plain sphere number of $D$. That is, $\mathcal{L}$ gives a coloring sequence for $D$ starting from $n$ seeds. After a small perturbation if needed, we may assume that the pairwise intersections of circles in $\mathcal{L}$ are disjoint from $D$. We define the {\it complexity} of $\mathcal{L}$ to be the (unsigned) count of intersection points between pairs of circles in $\mathcal{L}$: 

    \[
   \mathfrak{c}(\mathcal{L}):=  \sum_{L_i,L_j\in \mathcal{L}}\Big|L_i\bigcap_{i\neq j} L_j\Big|.
    \]
    Without loss of generality, we may assume that $\mathcal{L}$ has minimal complexity over all coloring sequences for $D$ starting with $n$ seeds. (Recall that $\rho(D)\leq \omega(D),$ and the Wirtinger number allows us to define a plain sphere coloring sequence where the loops are pairwise disjoint. Thus, it is clear that the complexity of a coloring sequence can be chosen to be zero after potentially increasing the number of seeds. What we show here is that $\mathcal{L}$ can be chosen to realize $\rho(D)$ while also satisfying $\mathfrak{c}(\mathcal{L})=0.$)

    If $\mathfrak{c}(\mathcal{L})\neq 0,$ there exist $L_i, L_j \in \mathcal{L}$ such that $i\neq j$ and $L_i\cap L_j \neq \emptyset$. Select $L_i$ with the property that it is the last loop in the ordered set $\mathcal{L}$ which has non-empty intersection with other loops in $\mathcal{L}$. That is, $\forall k>i,$ $L_k$ is disjoint from all other loops in $\mathcal{L}.$

We will prove the statement by contradiction; we will show that as long as $\mathfrak{c}(\mathcal{L})\neq 0$, the coloring sequence $\mathcal{L}$ does not in fact minimize complexity over all plain sphere coloring sequences starting with $n$ seeds. We will do this by producing a new plain sphere coloring sequence in which we replace $L_i$ with another circle $L_i^\ast$, such that $L_i^\ast$ defines a valid plain sphere coloring move at stage $i$ and colors the same strand as $L_i$. (This implies that all subsequent moves $L_{i+1}, \dots, L_{|s(D)|-n}$ in $\mathcal{L}$ can be performed without modification.) Lastly, we will show that  $\mathfrak{c}\bigl((L_1, \dots, L_{i-1}, L_i^*, L_{i+1},\dots,L_{|s(D)|-n})\bigr)< \mathfrak{c}(\mathcal{L})$.

In order to simplify the exposition, assume from now on that the coloring moves determined by $L_1, \dots, L_{i-1}$ have already been performed, and we are at stage $i$ of the coloring process, at which point the loop $L_i$ determines a valid plain sphere coloring move. 

    Denote by $E_i$ the disk bounded by $L_i$ in the plane of projection, and denote by $\mathcal{A}$ the intersection between the interior of $E_i$ and loops in $\mathcal{L}$:
    \[
    \mathcal{A}:= \mathring E_i \bigcap \Big(\bigcup_{{L_j\in\mathcal{L}}}L_j\Big).
    \]
    Note that $\mathcal{A}$ is the union of properly embedded $1$-manifolds (i.e. arcs and loops) in $E_i$. 
    The boundaries of the arcs in $\mathcal{A}$ cut $L_i=\partial E_i$ into a collection of arcs $\mathcal{A'}$, and $\mathcal{A}$ cuts $E_i$ into a collection of planar surfaces $\mathcal{P}$. For such a surface, a component of its boundary is either the union of subarcs of $\mathcal A$, or a single loop which is disjoint from all other loops in $\mathcal L$. 
    The intersection $L_i \cap D$ consists of two or more isolated points, and because loop $L_i$ defines a valid coloring move at stage $i$ of the coloring process, exactly one of these intersection points is uncolored. 
    Denote this point by $x$ and the arc in $\mathcal{A'}$ containing $x$ by $a$. The arc $a$ is contained in the boundary of some surface, $E_i^*$, in $\mathcal P$. Denote by $L_i^*$ the component of $\partial E_i^*$ which contains $a$. See Figure~\ref{fig:intersecting_loops}.
    Consider loop $L_j$, $j\neq i$, which has non-empty intersection with another loop in $\mathcal L$. By our choice of $L_i$, the loop $L_j$ precedes $L_i$ in the coloring sequence; so at stage $i$, all intersection points $L_j \cap D$ are colored. (On the other hand, if $L_j$ is disjoint from all other loops in $\mathcal L$, then points $L_j \cap D$ are not necessarily colored.)
    The arcs that make up $L_i^*$ are subarcs of $\mathcal A \cup a$ which lie on loops in $\mathcal L$ that non-trivially intersect other loops. Hence, all intersection points $L_i^* \cap D$ except for $x$ are colored at stage $i$, and $L_i^*$ defines a valid coloring move at this stage.

     Thus,
     \[
     \mathcal{L}^\ast:= \left(L_1, \dots, L_{i-1}, L_i^\ast, L_{i+1}, \dots, L_{|s(D)|-n}\right)
     \]
     is a valid plain sphere coloring sequence starting from $n$ seeds. In order to satisfy transversality conditions, from this point on we will use $L_i^\ast $ to denote the desired boundary component of a slightly shrunken copy of the planar surface $E_i^*$.
     
     Lastly, we check that $\mathfrak{c}(\mathcal{L}^\ast)<\mathfrak{c}(\mathcal{L}).$ The loop $L_i^\ast$ is disjoint from $\mathcal{L}$ by construction, while the loop $L_i$ has the property that $L_i\cap L_j \neq \emptyset$ for some $j\neq i.$ Therefore, as claimed, replacing $L_i$ by $L_i^\ast$ gives a plain sphere coloring sequence $\mathcal L^\ast$ which has strictly lower complexity than $\mathcal L$.
\end{proof}

\begin{figure}
    \centering

    \begin{tikzpicture}
        [x=1mm,y=1mm, ultra thick, 
        help lines/.style={very thin, draw=white}]
        \definecolor{lightpurple}{HTML}{DF73FF}
        \definecolor{darkpurple}{HTML}{9932CC}
        \definecolor{darkblue}{HTML}{003366}
        \definecolor{pink1}{HTML}{FF1493}

        \draw[white] (-40,0) -- (40,0);
        \draw[name path=c0, very thick] (0,0) circle (25);
        \draw[name path=int1, xshift=2.5mm,yshift=8.93mm] (0:0) -- (285.9:12);
        \draw[name path=curve, help lines] (10,5) .. controls (-7,-1) and (-23,-1) .. (-40,5) node[pos=0.5, below] {$s$};
        \draw[name path=int2, xshift=-10mm,yshift=6.7mm] (0:0) -- (275.3:12);
        \draw[lightpurple, xshift=-10mm,yshift=6.7mm, draw opacity=0.6, line width =2.5pt] (95.3:0.25) -- (275.3:12.25);
        \draw[name path=int3, xshift=-32.5mm, yshift=8.93mm] (0:0) -- (254.1:12);

        \path[name intersections={of=curve and int1,by=a'}];
        \path[name intersections={of=curve and int2,by=x}];
        \path[name intersections={of=curve and int3,by=b'}];
        \path[name intersections={of=curve and c0,by=p}];

        \draw[name path=ca, thick, dotted] (a') circle (8pt) node[below right=2mm] {$c$};
        \draw[name path=cb, white] (b') circle (8pt);

        \filldraw[white] (x) circle (6pt);
        \draw (10,5) .. controls (-7,-1) and (-23,-1) .. (-40,5);
        \filldraw[white] (a') circle (6pt);
        \filldraw[white] (b') circle (6pt);
        \draw[xshift=2.5mm,yshift=8.93mm] (0:0) -- (285.9:12);
        \draw[lightpurple, line width=2.5pt, xshift=2.5mm,yshift=8.93mm, draw opacity=0.6] (285.9:-0.25) -- (285.9:12.25);
        \draw[xshift=-32.5mm,yshift=8.93mm] (0:0) -- (254.1:12);

        \path[name intersections={of=ca and curve,by={ax,a}}];
        \path[name intersections={of=cb and curve,by={bx,b}}];
        \filldraw (p) circle (2pt) node[below left] {$x$};
        \filldraw (a) circle (2pt) node[above left] {$a$};
        \filldraw (b) circle (2pt) node[above right] {$b$};

        \draw[lightpurple, line width=2.5, draw opacity=0.6] ($(ax) + (-0.45,-0.15)$) -- (10.15,5.05);
        \draw[thick, dotted] (a') circle (8pt);

        \draw (97:21) -- (97:29)
        (70:21) -- (70:29)
        (300:21) -- (300:29);
        \draw[lightpurple, line width= 2.5pt, draw opacity=0.6] (97:20.75) -- (97:29.25)
        (70:20.75) -- (70:29.25)
        (300:20.75) -- (300:29.25);

        \node at (40:29) {$L_i$};
        \node at (-30:15) {$E_i$};
        
    \end{tikzpicture}
    
    \caption{Exactly one uncolored strand $s$ intersects the loop $L_i \in \mathcal L$ at stage $i$.}
    \label{fig:strand_labels}

\end{figure}
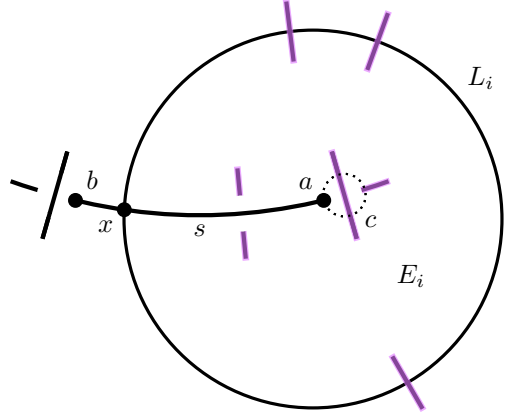

\begin{lemma} \label{lem:innermost-loops} Let $D$ be a diagram of a link $L$ with $\rho(D)=n$. There exists a diagram $D'$ of $L$ such that the conclusion of Lemma~\ref{lem:disjoint-loops} holds, and such that there exists a set of $n$ seeds for $D'$ and a valid coloring sequence $\mathcal{L'}$ for $D'$, starting from those seeds, with the property that all innermost loops of $\mathcal{L'}$ are Wirtinger loops.
\end{lemma}

\begin{proof}
    Since $\rho(D)=n,$ there exists a coloring sequence $\mathcal{L}$ for $D$ starting from $n$ seeds. Using Lemma~\ref{lem:disjoint-loops}, we can assume that the circles contained in $\mathcal{L}$ are disjoint, possibly after modification. 
    
    Let $L_i$ be an innermost loop in $\mathcal{L}$. This means that $L_i$ bounds a disk $E_i$ whose interior is disjoint from $\mathcal{L}$. As before, we consider $L_i$ at stage $i$ of the coloring process. That is, among the points in the intersection $L_i\cap D,$ only one point, denoted $x$, is not colored. Let $s$ be the strand of $D$ containing $x$ and  let $\partial(s)=\{a, b\}.$ By the definition of a plain sphere coloring move, $L_i$ intersects $s$ transversely in exactly one point. Therefore, $a$ and $b$ are contained in two different components of $\mathbb{R}^2\backslash L_i$. Without loss of generality, assume $a\in E_i, b\notin E_i.$ Denote by $c$ the crossing in $D$ incident to $a$. See Figure~\ref{fig:strand_labels}. The key fact here is that $s\cap E_i$ is the only arc in $D\cap E_i$ which is not colored at stage $i$. This is a consequence of the following two observations: $x$ is the only non-colored point in $\partial E_i\cap D$, so all other strands intersecting $\partial E_i$ are colored; and $L_i$ is innermost, so no strand entirely contained in $\mathring E_i$ can be colored via another circle in $\mathcal{L}$. That is, no strand in $\mathring E_i$ can change color after stage $i$. So, indeed, at stage $i$, of all components of $D\cap E_i$ only $s\cap E_i$ is not colored.

    {\it Case 1.} Assume that $s$ is not the overstrand at $c$. From the above discussion, we know that the overstrand and second understand at $c$ are both colored before stage $i$. Therefore, we can color $s$ via a Wirtinger move at the crossing $c$. In particular, we can replace the loop $L_i$ in $\mathcal{L}$ by a circle bounding a small neighborhood of $c$ in the plane. This defines a valid coloring sequence $\mathcal{L}'$ in which the innermost loop $L_i$ was replaced by one which satisfies the conclusions of the lemma.

    {\it Case 2.} Assume that $s$ is the overstrand at $c$. We also know that $c$ is the crossing where $s$ terminates. Hence, $s$ is both the overstrand and an understrand at $c$. In $P(D)$, the planar projection determined by $D$, $s$ is the union of edges of $P(D)$ and, since $s$ is both the overstrand and understrand at $c$, we can conclude that the union of a subset of the edges in $P(D)$ constitutes an embedded loop $\gamma$ in $E_i$. Moreover, $\gamma$ bounds a disk $G\subset E_i$ such that $\gamma=\partial G$ is contained in the projection of the strand $s$. Hence, there is an isotopy of $D$ along $G$ creating a diagram $D'$ in which crossing $c$ has been resolved in the direction which preserves the number of components of the link. Note that the image of the strand $s$ under this isotopy is contained in a strand of $D'$ that intersects $L_i$ in a second, necessarily colored, point. This eliminates the move determined by $L_i$ from the coloring sequence.

    One of the two cases will apply to each innermost loop in $\mathcal{L}$. Therefore, each innermost loop $L_i$ may be replaced by a Wirtinger loop or removed after an isotopy contained in $E_i$. The result is a coloring sequence starting from $n$ seeds in which all innermost loops represent Wirtinger moves.
\end{proof}

\begin{definition}
    Let $D$ be a partially colored link diagram, and let $\mathcal{L}$ be an ordered set of disjoint circles transverse to $D$ and representing valid coloring moves on $D$. Let loop $C$ be in $\mathcal L$ and let $G$ be the disk bounded in the plane by $C$. A loop $C$ is {\it depth-two} if: there exists at least one loop in $\mathcal{L}$ which is contained in the interior of $G$; and all loops contained in the interior of $G$ are innermost loops of $\mathcal{L}$. For $n>2$, the definition is inductive in the natural way: a loop $C$ is {\it depth-$n$} if disk $G$ contains at least one depth-$(n{-}1)$ loop and no nested sequence of $n$ or more loops. 
\end{definition}

\begin{figure}
    \centering

    \begin{tikzpicture}
        [x=1mm,y=1mm, ultra thick, 
        help lines/.style={very thin, draw=white}]
        \definecolor{lightpurple}{HTML}{DF73FF}
        \definecolor{darkpurple}{HTML}{9932CC}
        \definecolor{darkblue}{HTML}{003366}
        \definecolor{pink1}{HTML}{FF1493}

        \draw[name path=c0, very thick] (0,0) circle (25);
        \draw[name path=c1, very thick] (75,0) circle (25);

        \draw[name path=seed] (-18,7) .. controls (-5,10) and (5,10) .. (18,7) node[midway] (x) {} node[pos=0.7,above] {$s$};
        \draw[name path=apple, xshift=-13mm, yshift=4mm] (0:0) -- (100:8);
        {\draw[name path=banana, xshift=13mm, yshift=4mm] (0:0) -- (80:8);}

        \path[name intersections={of=seed and apple,by=a}];
        \path[name intersections={of=seed and banana,by=b}];

        \filldraw[name path=crossing1, white] (a) circle (5pt);
        \filldraw[name path=crossing2, white] (b) circle (5pt);
        \draw[xshift=-13mm,yshift=4mm] (0:0) -- (100:8);
        \draw[xshift=13mm,yshift=4mm] (0:0) -- (80:8);

        \draw[name path=alpha, thick, dashed] (0,-25) -- (x) node[pos=0.8,right] {$\alpha$};
        \filldraw (0,-25) circle (1.5pt) node[below left] {$\scriptstyle b$};
        \filldraw (x) circle (1.5pt) node[above left] {$\scriptstyle a$};
        \draw[name path=line1] (-8,-5) -- (8,-5) node[pos=0,left] {$s_1$};
        \draw[name path=line2] (-8,-10) -- (8,-10) node[left,pos=0] {$s_2$};
        \draw[name path=linem] (-8,-18) -- (8,-18) node[left,pos=0] {$s_m$};
        \node at (40:29) {$L_i$};

        \path[name intersections={of=alpha and line1,by=p1}];
        \path[name intersections={of=alpha and line2,by=p2}];
        \path[name intersections={of=alpha and linem,by=pm}];
        \filldraw (p1) circle (1.5pt) node[above left] {$\scriptstyle p_1$};
        \filldraw (p2) circle (1.5pt) node[above left] {$\scriptstyle p_2$};
        \filldraw (pm) circle (1.5pt) node[above left] {$\scriptstyle p_m$};

        \path[name intersections={of=crossing1 and seed,by={endpt1,endpt1'}}];
        \path[name intersections={of=crossing2 and seed,by={endpt2,endpt2'}}];
        \draw[line width=2.5pt, lightpurple, draw opacity=0.6] (endpt1) .. controls +(8,1.39) and ++(-8,1.39) .. (endpt2);

        \begin{scope}[xshift=75mm]
        \draw[name path=seed] (-18,7) .. controls (-5,10) and (5,10) .. (18,7) node[pos=0.7,above] {$s$};
        \draw[name path=apple, xshift=-13mm, yshift=4mm] (0:0) -- (100:8);
        {\draw[name path=banana, xshift=13mm, yshift=4mm] (0:0) -- (80:8);}

        \path[name intersections={of=seed and apple,by=a}];
        \path[name intersections={of=seed and banana,by=b}];

        \filldraw[white] (a) circle (5pt);
        \filldraw[white] (b) circle (5pt);
        {\draw[xshift=-13mm,yshift=4mm] (0:0) -- (100:8);}
        {\draw[xshift=13mm,yshift=4mm] (0:0) -- (80:8);}

        \draw[name path=l1] (-8,-5) -- (8,-5);
        \draw[name path=l2] (-8,-10) -- (8,-10);
        \draw[name path=l3] (-8,-18) -- (8,-18);
        \draw[line width=2.5pt, lightpurple, draw opacity=0.6] ($(endpt1)+(75,0)$) .. controls +(8,1.39) and ++(-8,1.39) .. ($(endpt2)+(75,0)$);
        \draw[line width=5pt, white] (-6.9,8.9) -- (6.9,8.9);
        \draw[name path=seed1] (-7,8.87) .. controls (0,10.1) and (-6.5,-29) .. (0,-29);
        \draw[name path=seed2] (7,8.87) .. controls (0,10.1) and (6.5,-29) .. (0,-29);

        \path[name intersections={of=seed1 and l1,by=1}];
        \path[name intersections={of=seed1 and l2,by=2}];
        \path[name intersections={of=seed1 and l3,by=3}];
        \path[name intersections={of=seed2 and l1,by=4}];
        \path[name intersections={of=seed2 and l2,by=5}];
        \path[name intersections={of=seed2 and l3,by=6}];

        \filldraw[white] (1) circle (5pt);
        \draw[dotted, thick] (1) circle (6.5pt) node[above left=1.5mm] {$\scriptstyle q_{1,1}$};
        \filldraw[white] (2) circle (5pt);
        \draw[dotted, thick] (2) circle (6.5pt) node[above left=1.5mm] {$\scriptstyle q_{2,1}$};
        \filldraw[white] (3) circle (5pt);
        \draw[dotted, thick] (3) circle (6.5pt) node[above left=1.5mm] {$\scriptstyle q_{m,1}$};
        \filldraw[white] (4) circle (5pt);
        \draw[dotted, thick] (4) circle (6.5pt) node[above right=1.5mm] {$\scriptstyle q_{1,2}$};
        \filldraw[white] (5) circle (5pt);
        \draw[dotted, thick] (5) circle (6.5pt) node[above right=1.5mm] {$\scriptstyle q_{2,2}$};
        \filldraw[white] (6) circle (5pt);
        \draw[dotted, thick] (6) circle (6.5pt) node[above right=1.5mm] {$\scriptstyle q_{m,2}$};
        \draw (-7,8.87) .. controls (0,10.1) and (-6.5,-29) .. (0,-29);
        \draw (7,8.87) .. controls (0,10.1) and (6.5,-29) .. (0,-29);

        \draw[line width=2.5pt, lightpurple, draw opacity=0.6] (-7,8.87) .. controls (0,10.1) and (-6.5,-29) .. (0,-29);
        \draw[line width=2.5pt, lightpurple, draw opacity=0.6] (7,8.87) .. controls (0,10.1) and (6.5,-29) .. (0,-29);

        \node at (40:29) {$L_i$};
        \end{scope}

        \draw[->, line width=2pt] (32.5,0) -- (42.5,0);
    \end{tikzpicture}
    
    \caption{An isotopy to ensure no seed strand $s$ is contained entirely in the interior of any disk $E_i$ in the plane with $\partial(E_i)=L_i \in \mathcal L$. After a sequence of $m$ Reidemeister II moves, $2m$-many Wirtinger coloring loops are added in the appropriate positions to sequence $\mathcal L$ in order to obtain $\mathcal L'$.}
    \label{fig:seed_isotopy}
\end{figure}
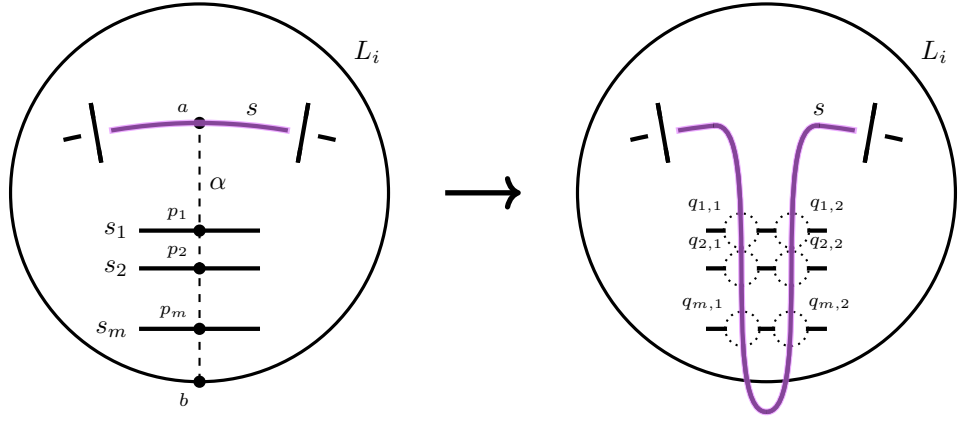

\begin{lemma} \label{lem:no-seeds-inside} Let $D$ be a diagram of a link $L$ with $\rho(D)=n$. There exists a diagram $D'$ of $L$ and a valid coloring sequence $\mathcal{L}'$ for $D'$ starting from $n$ seeds, such that all the conclusions of Lemma~\ref{lem:innermost-loops} are satisfied and, moreover, if $L_i$ is any loop in $\mathcal{L}'$, then no seed for the coloring sequence is entirely contained in the interior of the disk $E_i$ bounded by $L_i$ in the plane of projection.
\end{lemma}

\begin{proof}
    By Lemma \ref{lem:innermost-loops} we can assume that $\rho(D)=n$ and that $\mathcal{L}$ is a plain sphere coloring sequence starting with seeds $s_1, s_2, ..., s_n$ such that all intermost loops in $\mathcal{L}$ are Wirtinger loops. If $L_i\in \mathcal{L}$ represents a Wirtinger move, then without loss of generality we may assume that $L_i$ is the boundary of a small neighborhood of a crossing in $D$, and the desired conclusion automatically holds for $L_i$. In general, however, $E_i$ may contain multiple crossings of $D$ and potentially entire strands. Assume that a seed strand $s$ is contained in $\mathring E_i$. Define $\alpha$ to be an embedded arc in $E_i$, transverse to $D$, disjoint from small neighborhoods of crossings of $D$, and with the property that $\partial\alpha=:\{a, b\}$ has $a\in \mathring s$ and $b\in (L_i\backslash D)$. See Figure~\ref{fig:seed_isotopy}. We may use $\alpha$ to guide an isotopy of $s$ in the plane of projection to produce a new strand such that the image of $a$ after the isotopy lies outside of $E_i$, in a small neighborhood of $b$. The isotopy can be thought of as a sequence of Reidemeister II moves in which $s$ passes over every strand it encounters along $\alpha$.  

    Let $D_1$ be the diagram that results from the above isotopy. We claim that $\rho(D_1)=n$ and that $D_1$ admits a coloring sequence which is obtained from $\mathcal{L}$ by inserting $2m$ additional loops representing Wirtinger moves, where $m:=|\alpha \cap D|$. More specifically, let $\alpha\cap D=:\{p_1, \dots, p_m\}.$ The additional Wirtinger loops will be interspersed across $\mathcal{L}$ according to when the strands in $D$ containing the $p_j$ are colored. See Figure~\ref{fig:seed_isotopy} for the isotopy of $s$ and the new strands and crossings created.
    
    Note that each point $p_j$ corresponds to a Reidemeister II move performed while isotoping $s$ along $\alpha$. Thus, in a neighborhood of  $p_j$, two new crossings---call them $q_{j, 1}$ and $q_{j, 2}$---are created during the isotopy, and at both crossings $s$ is the overstrand. Let $s_j$ denote the strand of $D$ containing $p_j$. (After the isotopy, $s_j$ is subdivided into three strands by $q_{j, 1}$ and $q_{j, 2}$.) At some stage in the coloring sequence $\mathcal{L}$ for $D$, the strand $s_j$ is colored. This implies that, at the corresponding stage of the coloring sequence for $D_1,$ one of the understrands at either $q_{j, 1}$ or $q_{j, 2}$ is colored. But the overstrand at these crossings, $s$, is a seed, so it is colored as well. Thus, we can perform two consecutive Wirtinger moves at these crossings and color all three strands into which $s_j$ has been subdivided. As a result, any further coloring moves given by $\mathcal L$ whose validity relies on $s_j$ being colored will be valid. Repeating this procedure for each $j\in \{1, \dots, m\}$ produces the desired coloring sequence for~$D_1$.

    In sum, after an isotopy supported in a neighborhood of an arc, we ensured that the seed strand $s$ is no longer contained in the interior of $E_i$. Moreover, we produced a coloring sequence, with the same number of seeds, for the diagram resulting from this isotopy. Repeating this procedure for every instance where a seed is contained entirely within the interior of one of the disks $E_i,$ we arrive at the desired diagram $D'$ and the valid coloring sequence $\mathcal{L}'$ from $n$ seeds.
\end{proof}

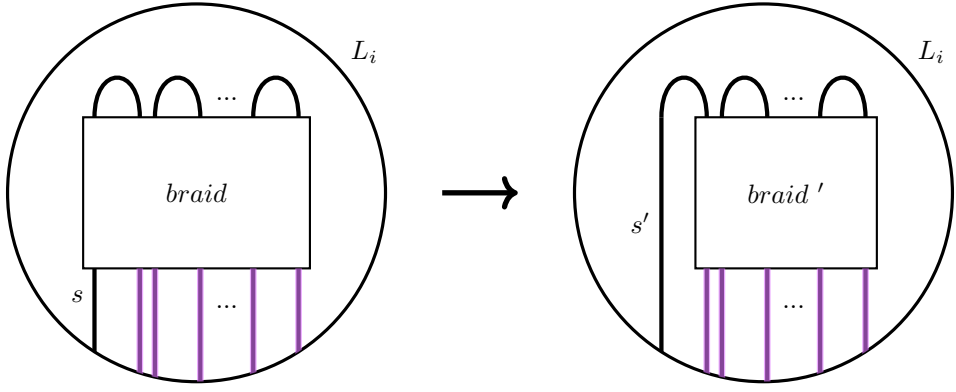
\begin{figure}
    \centering

    \begin{tikzpicture}
        [x=1mm,y=1mm, ultra thick, 
        help lines/.style={very thin, draw=white},
        colored/.style={line width=2.5pt, draw=lightpurple, draw opacity=0.6}]
        \definecolor{lightpurple}{HTML}{DF73FF}
        \definecolor{darkpurple}{HTML}{9932CC}
        \definecolor{darkblue}{HTML}{003366}
        \definecolor{pink1}{HTML}{FF1493}

        \draw[name path=c0, very thick] (0,-5) circle (25);
        \draw[name path=c1, very thick] (75,-5) circle (25);

        \begin{scope}
        \clip (0,-5) circle (25);

        \draw[thick] (-15,-15) rectangle (15,5);
        \draw (-13.5,5) .. controls (-13.5,12) and (-7.5,12) .. (-7.5,5);
        \draw (-5.5,5) .. controls (-5.5,12) and (0.5,12) .. (0.5,5);
        \draw (7.5,5) .. controls (7.5,12) and (13.5,12) .. (13.5,5);

        \draw (-13.5,-15) -- (-13.5,-30) node[pos=0.25, left] {$s$};
        \draw (-7.5,-15) -- (-7.5,-30);
        \draw (-5.5,-15) -- (-5.5,-30);
        \draw (0.5,-15) -- (0.5,-30);
        \draw (7.5,-15) -- (7.5,-30);
        \draw (13.5,-15) -- (13.5,-30);

        \draw[colored] 
        (-7.5,-15) -- (-7.5,-30)
        (-5.5,-15) -- (-5.5,-30)
        (0.5,-15) -- (0.5,-30)
        (7.5,-15) -- (7.5,-30)
        (13.5,-15) -- (13.5,-30);
        
        \node at (4,-20) {...};
        \node at (4,7.5) {...};
        \node at (0,-5) {$braid$};
        \end{scope}
        
        \node at ($(40:29)+(0,-5)$) {$L_i$};

        \begin{scope}[xshift=75mm]
        \clip (0,-5) circle (25);

        \draw[thick] (-9,-15) rectangle (15,5);
        \draw (-13.5,5) .. controls (-13.5,12) and (-7.5,12) .. (-7.5,5);
        \draw (-5.5,5) .. controls (-5.5,12) and (0.5,12) .. (0.5,5);
        \draw (7.5,5) .. controls (7.5,12) and (13.5,12) .. (13.5,5);

        \draw (-13.5,5) -- (-13.5,-30) node[left, pos=0.4] {$s'$};
        \draw (-7.5,-15) -- (-7.5,-30);
        \draw (-5.5,-15) -- (-5.5,-30);
        \draw (0.5,-15) -- (0.5,-30);
        \draw (7.5,-15) -- (7.5,-30);
        \draw (13.5,-15) -- (13.5,-30);

        \draw[colored] (-7.5,-15) -- (-7.5,-30)
        (-5.5,-15) -- (-5.5,-30)
        (0.5,-15) -- (0.5,-30)
        (7.5,-15) -- (7.5,-30)
        (13.5,-15) -- (13.5,-30);
        
        \node at (4,-20) {...};
        \node at (4,7.5) {...};
        \node at (3,-5) {$braid \ '$};
        \end{scope}
        \node at ($(40:29)+(75,-5)$) {$L_i$};

        \draw[->, line width=2pt] (32.5,-5) -- (42.5,-5);
    \end{tikzpicture}
    
    \caption{An isotopy at stage $i$ of the coloring process of a rational $m$-strand tangle supported in disk $E_i$ with $\partial(E_i)=L_i$. After isotopy, all strands that meet $E_i$ can be colored using only Wirtinger moves, starting from $2m{-}1$ colored strands which intersect $L_i$.}
    \label{fig:braid_isotopy}
\end{figure}

Before presenting the next lemma, we review some standard properties of tangles. An \emph{$m$-strand tangle} is a collection of $m$ disjoint arcs properly embedded in a $3$-ball. An $m$-strand tangle is \emph{rational} if all arcs of the tangle can be simultaneously isotoped into the boundary of the $3$-ball. 

\begin{remark}\label{rem:tangle_diagram} Given a tangle $R$ properly embedded in $D^2\times [0,1]$ such that $\partial R\subset (\partial D^2) \times [0,1]$, $R$ has a tangle diagram, analogous to a knot diagram, achieved by projecting $R$ onto $D^2\times \{0\}$. Let $x\in D^2$ denote the point with coordinates (0, 1), where $D^2$ is identified with the unit disk in the plane. If $R$ is a rational tangle embedded in $D^2\times [0,1]$ such that $\partial R\subset \partial D^2 \times [0,1]$, then after planar isotopy, and possibly Reidemeister moves supported in the interior of $D^2\times \{0\}$, $R$ has a tangle diagram as in the left image in Figure \ref{fig:braid_isotopy}. In brief, this is true since the disks of parallelism for the arcs in $R$ can all be isotoped to be disjoint from the disk that is a regular neighborhood of $(D^2\times \{1\})\cup (D^2\times \{0\})\cup (\{x\}\times [0,1])$ in $\partial(D^2\times [0,1])$. Isotoping the arcs of $R$ along these disks of parallelism until they almost lie in $\partial(D^2\times [0,1])$ results in a diagram of $R$ that lies in an annular neighborhood of $\partial D^2\times \{0\}$ in $D^2\times \{0\}$ such that the projection of $R$ is disjoint from a neighborhood of $\{x\}\times \{0\}$ and each arc of the projection has exactly one maximum with respect to the radial height function on $D^2\times \{0\}$ (where we think of height as increasing as we approach the center of the disk). After a planar isotopy, this diagram becomes a diagram as in Figure \ref{fig:braid_isotopy}, left. \end{remark}

\begin{remark}\label{rem:minima_rational} Any tangle $R$ properly embedded in $D^2\times [0,1]$ such that each arc of $R$ has a unique local maximum (or each has a unique local minimum) with respect to projection on the $[0,1]$ component of $D^2\times [0,1]$ is a rational tangle. In brief, this is true since if an arc $\alpha$ in $R$ has a unique local maximum at height $a\in [0,1]$, then there is a disk of parallelism between the subarc of $\alpha$ above $D^2\times \{a-\varepsilon\}$ and an arc in $D^2\times \{a-\varepsilon\}$. Since every arc in $R$ contains a unique local maximum, this disk can be extended downward by moving both endpoints along the arc until it becomes a disk of parallelism between $\alpha$ and a subarc in $\partial(D^2\times [0,1])$. After generating a disk of parallelism in this way for each arc of $R$, we can ensure that these disks are pairwise disjoint by an application of a standard innermost disk and outermost arc argument. Thus, $R$ is rational.\end{remark}

\begin{lemma}\label{lem:eliminate-depth-two} Let $D$ be a diagram of a link $L$ with $\rho(D)=n$, and let $\mathcal{L}=(L_1, L_2, \dots, L_{|s(D)|-n})$ be a plain sphere coloring sequence for $D$, satisfying the conclusions of Lemma~\ref{lem:no-seeds-inside}. Let $L_i$ be a depth-two loop in $\mathcal{L}$ and denote by $E_i$ the disk in the plane of projection bounded by $L_i$. There exists an isotopy of $D,$ supported in a small neighborhood of $E_i$, such that the resulting diagram $D'$ has a coloring sequence $\mathcal{L}'$ with $n$ seeds and such that $L_i$ is replaced by a collection of innermost loops.
\end{lemma}

\begin{proof}  
    By assumption, $\mathcal{L}$ is a coloring sequence for $D$ starting from $n$ seeds and satisfying the conclusions of Lemma \ref{lem:no-seeds-inside}. Since $L_i$ is a depth-two loop, there exists at least one loop of $\mathcal{L}$ contained in $E_i$, and moreover any such loop is innermost. Since the coloring sequence satisfies the conclusions of Lemma~\ref{lem:innermost-loops}, all loops contained in $E_i$ are Wirtinger loops. Since the conclusions of Lemma~\ref{lem:no-seeds-inside} hold as well, we also have that no seed for the coloring sequence $\mathcal{L}$ is entirely contained in $E_i$. Let $s$ be the unique strand of $D$ which intersects $L_i$ and which is colored at stage $i$ in the coloring process. 
 
     Recall that $P:\mathbb{R}^3\rightarrow \mathbb{R}^2$ denotes the standard projection map and that $P(L)$ is the regular projection that gives rise to $D$. We can assume that $L\subset \mathbb{R}^2 \times [0,1]$. If $L_i\cap D=2m$, then, since no seed is entirely contained in $E_i$, $E_i\times [0,1]=P^{-1}(E_i)\cap (\mathbb{R}^2 \times [0,1])$ is a $3$-ball containing a properly embedded $m$-strand tangle $T=L\cap (E_i\times [0,1])$. In particular, the lack of seeds in $E_i$ implies $T$ contains no simple closed curves. Double $E_i\cap D$ along $\partial E_i$ to obtain a link diagram $D^{dbl}$ of the link $L^{dbl}$, where $L^{dbl}$ is obtained by doubling $T$ along $\partial E_i\times [0,1]$. Since $E_i$ does not entirely contain any seed strands and all loops contained in $E_i$ are Wirtinger loops, the set of $2m$ strands of $D^{dbl}$ that intersect $\partial E_i$, call them $\sigma_1,\dots, \sigma_{2m}$, are a collection of seed strands for a Wirtinger coloring of $D^{dbl}$. By Theorem 1.2 of~\cite{blair2019incompatibility}, there is an isotopy of $L^{dbl}$ achieved by modifying only the $z$-coordinate of points on $L^{dbl}$ and preserving the projection $P(L^{dbl})$ at all times, after which $L^{dbl}$ has exactly $2m$ local maxima with respect to projection to the $z$-axis, and each of these maxima projects to exactly one of $\sigma_1\cap \partial E_i,\dots, \sigma_{2m}\cap \partial E_i$. This isotopy restricts to a proper isotopy of $T$ in $E_i\times [0,1]$, after which each strand of $T$ contains a unique local minimum with respect to the $z$-axis. Hence, by Remark \ref{rem:minima_rational}, $T$ must be a rational $m$-strand tangle. Since $T$ is rational, then, by Remark \ref{rem:tangle_diagram}, there is a sequence  of Reidemeister moves supported in $E_i$ that produces a new diagram $D^*$ as in Figure~\ref{fig:braid_isotopy}, left. That is, $T\cap E_i$ is obtained from a braid by taking the plat closure to one side. Additionally, by taking the point $x$ in Remark \ref{rem:tangle_diagram} to be a point in $L_i$ which lies just clockwise of $s\cap L_i$, we can ensure that $s$ is the leftmost strand entering the braid box from below, as in Figure~\ref{fig:braid_isotopy}, left. Furthermore, it is well known that there is an isotopy of $T$ fixing $\partial T$ which pulls one arc out of the braid box, shown in Figure~\ref{fig:braid_isotopy}; see, for example, the (unique) Claim within the proof of Theorem 1 in~\cite{Schultens03}. After isotopy, call the resulting link diagram $D'$ (Figure~\ref{fig:braid_isotopy}, right) and let $s'$ be the strand containing the image of $s$. Notice that since all strands entering the braid box from below are colored at stage $i$, then all strands of $D'$ which meet $E_i$ can be colored via Wirtinger loops, even~$s'$. 
    
    Thus, after an isotopy supported in a neighborhood of $E_i,$ we can replace $L_i$ by a collection of Wirtinger loops. In particular, we eliminated a depth-two loop from the coloring sequence.
\end{proof}

\begin{corollary}\label{cor:l(D)=w(D)}
    If $L$ is a link and $D$ is a diagram of $L$ with $\rho(D)=\rho(L)=n$, then there exists a diagram $D',$ equivalent to $D$, such that $\rho(D')=\omega(D')=n$.
\end{corollary}

\begin{proof}
Let $D$ be a link diagram and $\mathcal{L}$ a plain sphere coloring sequence for $D$ starting with $n$ seeds.  By repeated application of Lemma~\ref{lem:eliminate-depth-two}, we can eventually eliminate all depth-two loops. Note that eliminating a single loop as described in the lemma may not immediately reduce the number of depth-two loops as it could simultaneously produce a new depth-two loop from a previous depth-three loop; however, the procedure can be iterated as many times as needed. 

After the necessary amount of persistence, we arrive at a diagram $D',$ equivalent to $D$, which admits a plain sphere coloring sequence $\mathcal{L}'$, also with $n$ seeds, such that $\mathcal{L}'$ contains no depth-two loops. Since $D$ was chosen to be minimal with respect to plain sphere number, then $\rho(D')=n$. But this in fact implies that all loops in $\mathcal{L}'$ are innermost. By Lemma~\ref{lem:innermost-loops}, all moves in the coloring sequence $\mathcal{L}'$ are Wirtinger moves. Therefore, $n\geq \omega(D')$. But we also have, by definition, that $\omega(D')\geq \rho(D')=n$. 
\end{proof}

\begin{proof}[Proof of Theorem~\ref{thm:main}] Let $L$ be a link in $S^3$ with $\rho(L)=n$. We wish to show that $\omega(L)=n$ as well. Let $D$ be a diagram realizing $\rho(L)$. This implies that $D$ admits a plain sphere coloring sequence starting from $n$ seeds. By Corollary~\ref{cor:l(D)=w(D)}, there exists a diagram $D'$ of $L$ which can be colored starting with $n$ seeds and using only Wirtinger moves. Therefore, $n\geq \omega(D')\geq \omega(L)\geq \rho(L)=n.$ It then follows from~\cite{blair2020wirtinger} that $\beta(L)=\omega(L)=n.$
\end{proof}

\section{Example and computations} \label{sec:Examples}

\subsection{Example} We compute the plain sphere number of a minimal diagram of the knot $K=14n_{1527}$, shown in Figure~\ref{fig:14n1527_colored}. The bridge number of $K$ is 3 and equals the plain sphere number of the pictured diagram. The Wirtinger number of the diagram is 4. We thank Nathan Dunfield for detecting this example and sharing it with us.

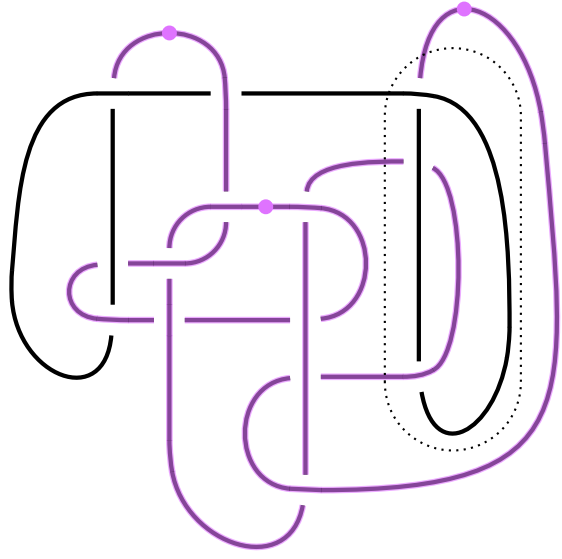
\begin{figure}[h]
    \centering

        \begin{tikzpicture}
        [x=1.5mm,y=1.5mm, ultra thick, 
        help lines/.style={very thin, draw=white},
        crossing/.style={circle, draw, white, minimum size=10pt, node font=\tiny},
        colored/.style={line width=2.5pt, draw=lightpurple, draw opacity=0.6}]
        \definecolor{lightpurple}{HTML}{DF73FF}
        \definecolor{darkpurple}{HTML}{9932CC}
        \definecolor{darkblue}{HTML}{003366}
        \definecolor{pink1}{HTML}{FF1493}

        \draw[name path=s1, help lines] (-10,0) -- (-10,7);
        \draw[name path=s2, help lines] (-1.5,-5) -- (-1.5,-15);
        \draw[name path=s3, help lines] (16,3) -- (16,8);

        \node[crossing] (1) at (-15,0) {};
        \node[crossing] (2) at (-5,0) {} 
            edge[name path=seed1, bend right=85, distance=8mm] (1)
            edge[colored, bend right=85, distance=8mm] (1)
            edge (1);
        \node[crossing] (3) at (12,0) {}
            edge (2);
        \node[crossing] (b) at (23,-3) {}
            edge[name path=seed3, out=100, in=85, distance=15mm] (3)
            edge[colored, name path=seed3, out=100, in=85, distance=15mm] (3);
        \node[crossing] (4) at (12,-6) {};
        \node[crossing] (5) at (-5,-10) {} 
            edge (2)
            edge[colored] (2);
        \node[crossing] (6) at (2,-10) {} 
            edge[name path=seed2] (5)
            edge[colored] (5)
            edge[out=85, in=180, distance=4mm] (4)
            edge[colored, out=85, in=180, distance=4mm] (4);
        \node[crossing] (7) at (-15,-15) {};
        \node[crossing] (8) at (-10,-15) {} 
            edge[bend right=45] (5)
            edge[colored, bend right=45] (5)
            edge (7)
            edge[colored] (7)
            edge[bend left=45] (5)
            edge[colored, bend left=45] (5);
        \node[crossing] (9) at (-15,-20) {} 
            edge[bend left=85, distance=5mm] (7)
            edge[colored, bend left=85, distance=5mm] (7)
            edge (1);
        \node[crossing] (d) at (-24,-17) {}
            edge[out=85, in=180] (1)
            edge[out=275, in=265, distance=10mm] (9);
        \node[crossing] (10) at (-10,-20) {} 
            edge (9)
            edge[colored] (9)
            edge (8)
            edge[colored] (8);
        \node[crossing] (11) at (2,-20) {} 
            edge (10)
            edge[colored] (10)
            edge[bend right=85, distance=8mm] (6)
            edge[colored, bend right=85, distance=8mm] (6);
        \node[crossing] (a) at (20,-22) {}
            edge[out=90, in=350, distance=10mm] (3);
        \node[crossing] (12) at (2,-25) {};
        \node[crossing] (13) at (12,-25) {} 
            edge[bend right=65, distance=5mm] (4)
            edge[colored, bend right=65, distance=5mm] (4)
            edge (12)
            edge[colored] (12)
            edge (3)
            edge[out=280, in=260, distance=10mm] (a);
        \node[crossing] (14) at (2,-35) {} 
            edge (6)
            edge[colored] (6)
            edge[bend left=85, distance=8mm] (12)
            edge[colored, bend left=85, distance=8mm] (12)
            edge[out=0, in=275, distance=35mm] (b)
            edge[colored, out=0, in=275, distance=35mm] (b);
        \node[crossing] (c) at (-10,-32) {}
            edge (10)
            edge[colored] (10)
            edge[out=280, in=260, distance=10mm] (14)
            edge[colored, out=280, in=260, distance=10mm] (14);

        \draw[shift=(1)] (0:6pt) -- (180:6pt);
        \draw[shift=(2)] (95:6pt) .. controls +(275:3pt) and ++(90:4pt) .. (270:6pt);
            \draw[colored, shift=(2)] (95:6pt) .. controls +(275:3pt) and ++(90:4pt) .. (270:6pt);
        \draw[shift=(3)] (180:6pt) .. controls +(0:4pt) and ++(170:4pt) .. (350:6pt);
        \draw[shift=(b)] (100:6pt) .. controls +(280:3pt) and ++(95:4pt) .. (275:6pt);
            \draw[colored, shift=(b)] (100:6pt) .. controls +(280:3pt) and ++(95:4pt) .. (275:6pt);
        \draw[shift=(5)] (180:6pt) -- (0:6pt);
            \draw[colored, shift=(5)] (180:6pt) -- (0:6pt);
        \draw[shift=(6)] (180:6pt) .. controls +(0:4pt) and ++(175:3pt) .. (355:6pt);
            \draw[colored, shift=(6)] (180:6pt) .. controls +(0:4pt) and ++(175:3pt) .. (355:6pt);
        \draw[shift=(8)] (180:6pt) -- (0:6pt);
            \draw[colored, shift=(8)] (180:6pt) -- (0:6pt);
        \draw[shift=(d)] (275:6pt) .. controls +(95:4pt) and ++(265:3pt) .. (85:6pt);
        \draw[shift=(9)] (175:6pt) .. controls +(355:3pt) and ++(180:4pt) .. (0:6pt);
            \draw[colored, shift=(9)] (175:6pt) .. controls +(355:3pt) and ++(180:4pt) .. (0:6pt);
        \draw[shift=(10)] (90:6pt) -- (270:6pt);
            \draw[colored, shift=(10)] (90:6pt) -- (270:6pt);
        \draw[shift=(a)] (90:6pt) .. controls +(270:4pt) and ++(80:3pt) .. (260:6pt);
        \draw[shift=(13)] (180:6pt) .. controls +(0:4pt) and ++(210:4pt) .. (25:6pt);
            \draw[colored, shift=(13)] (180:6pt) .. controls +(0:4pt) and ++(210:4pt) .. (25:6pt);
        \draw[shift=(c)] (90:6pt) .. controls +(270:4pt) and ++(100:3pt) .. (280:6pt);
            \draw[colored, shift=(c)] (90:6pt) .. controls +(270:4pt) and ++(100:3pt) .. (280:6pt);
        \draw[shift=(14)] (175:6pt) -- (0:6pt);
            \draw[colored, shift=(14)] (175:6pt) -- (0:6pt);

        \path[name intersections={of=seed1 and s1,by=dot1}];
        \path[name intersections={of=seed2 and s2,by=dot2}];
        \path[name intersections={of=seed3 and s3,by=dot3}];

        \filldraw[lightpurple] (dot1) circle (2pt)
            (dot2) circle (2pt)
            (dot3) circle (2pt);

        \draw[thick, dotted, rounded corners=9mm] (9,-31.5) rectangle (21,4);
        
    \end{tikzpicture}

    \caption{A partially colored minimal diagram $D$ of the knot $K=14n_{1527}$ with $\beta(K) =\rho(D)=3$ and $\omega(D)=4.$ Seed strands are marked with dots. Colored strands are purple. At this stage of the coloring sequence, only Wirtinger moves have been performed, and no further Wirtinger moves are possible. The dotted circle $L_8$ represents a loop coloring move. (The reader may find several alternative loop coloring moves.) After coloring the black strand intersecting $L_8$, the rest of the diagram can be colored by Wirtinger moves.}
    \label{fig:14n1527_colored}
\end{figure}

\subsection{Computing the plain sphere number of diagrams} \label{sec:compute}
Let $D$ be a link diagram in $S^2$, $P(D)$ be the projection corresponding to $D$. Recall that a region of $D$ is the closure of a connected component of $S^2\setminus P(D)$. Note that the boundary of a region is a cycle in $P(D)$ where $P(D)$ is viewed as a $4$-valent graph. Let $\Gamma_D$ denote the dual graph to $D$.

\begin{definition}
    
Let $L$ be an embedded loop in $S^2$ that is transverse to $D$ and is disjoint from neighborhoods of crossings of $D$. We say $L$ is \emph{tight} if for every region $R$ of $D$ such that $L\cap R\neq \emptyset$, $L\cap R$ is an arc properly embedded in $R$ with endpoints on distinct edges of $P(D)$ in $\partial R$.
\end{definition}

Every tight loop in $S^2$ is isotopic in $S^2$ to an embedded loop in $\Gamma_D$ via an isotopy that is transverse to the edges of $P(D)$ and disjoint from the vertices of $P(D)$.

\begin{lemma}
    Let $D$ be a link diagram. There exists a plain sphere coloring sequence that realizes $\rho(D)$ and consists entirely of tight loops.
\end{lemma}

\begin{proof}
Given a plain sphere coloring sequence $\mathcal{L}$ for a diagram $D$, define $$\tau(\mathcal{L})=\sum_{L\in \mathcal{L}}|L\cap D|.$$

Let $\mathcal{L}$ be a plain sphere coloring sequence such that $\mathcal L$ realizes $\rho(D)$ and the loops in $\mathcal{L}$ are disjointly embedded. Assume that $\mathcal{L}$ minimizes $\tau(\mathcal{L})$ over all such sequences. By Definition~\ref{def:loopcolor}, for every $L_i\in \mathcal{L}$ and for every region $R$ of $D$, if $L_i\cap R\neq \emptyset$, then $L_i\cap R$ is a collection of arcs.

Suppose toward a contradiction that there exists $L_i\in \mathcal{L}$ and a region $R$ of $D$ such that $L_i \cap R$ consists of two or more arcs. If $L_i$ colors the strand $s$ of $D$ at stage $i$ of the coloring process, then $L_i\cap s$ is a single point. Let $\gamma_1$ and $\gamma_2$ be two distinct arcs of $L_i \cap R$ that are contained in the boundary of the same component of $R\setminus L_i$. Denote the closure of this component by $C$. Let $\alpha$ be an arc properly embedded in $C$ with one endpoint in $\gamma_1$ and one endpoint in $\gamma_2$. Let $N$ be a rectangular $I$-fibered neighborhood of $\alpha$ in $C$ such that $\partial N = \beta_1 \cup \beta_2 \cup \delta_1 \cup \delta_2$, where $\beta_i$ is embedded in $\gamma_i$ and $\mathring \delta_i \subset \mathring C$. Surger $L_i$ along $\alpha$ in the standard way: by removing $\beta_1$ and $\beta_2$ and gluing in $\delta_1$ and $\delta_2$. This results in two loops in $S^2$, $L'_i$ and $L''_i$, both of which are transverse to $D$ and disjoint from neighborhoods of crossings in $D$. Both $L'_i$ and $L''_i$ have non-trivial intersection with $D$, so $|L'_i \cap D|<|L_i \cap D|$ and $|L''_i \cap D|<|L_i \cap D|$. Exactly one of $L'_i$ and $L''_i$ has non-trivial intersection with $s$, say $L'_i$ does. All other points of intersection of $D$ with $L'_i$ and $L''_i$ are contained in colored strands. Then $\mathcal{L}'=(L_1,\dots,L_{i-1},L'_i, L_{i+1},\dots,L_m)$ is a plain sphere coloring sequence realizing $\rho(D)$ such that $\tau(\mathcal{L}')<\tau(\mathcal{L})$, a contradiction. Hence, for $\mathcal L$ which minimizes $\tau$, there is no loop $L_i \in \mathcal L$ which intersects a region of $D$ in more than one arc.

It remains to show that for a loop $L_i \in \mathcal L$ which non-trivially intersects a region $R$, the arc $L_i \cap R$ has endpoints on distinct edges of $P(D)$. Suppose $L_i \cap R$ has endpoints on the same edge, $\alpha$, of the projection. Then there is a region $R'$ adjacent to $R$ which shares edge $\alpha$. The loop $L_i$ intersects $\alpha$ in two points, and we know that $L_i \cap R'$ is one arc. Then $L_i$ must be an embedded loop which is entirely contained in $R$ and $R'$ and which intersects $D$ in exactly two points. These two points lie on the same strand of $D$ so are either both colored or both uncolored. In both cases, $L_i$ is not a valid loop in the sense of Definition~\ref{def:loopcolor}. We conclude that $L_i \cap R$ has endpoints on distinct edges of $P(D)$, and that for a sequence $\mathcal L$ which realizes $\rho(D)$, if $\mathcal L$ minimizes $\tau$, then $\mathcal L$ consists entirely of tight loops.
\end{proof}

Consequently, $\rho(D)$ can be realized by a sequence of loops represented by embedded cycles in $\Gamma_D$, an approach suggested by Nathan Dunfield. In particular, computing $\rho(D)$ for a fixed link diagram $D$ is algorithmic. The forthcoming version 3.3 of SnapPy~\cite{SnapPy} will include a feature computing $\rho(D)$. A computation performed using this feature shows that over 600 diagrams in the knot table through 16 crossings have plain sphere number strictly less than Wirtinger number~\cite{Nathan2025email}.

We conclude by pointing out that our plain spheres generalize the connected sum spheres which appear in the study of the visual primeness of links; see~\cite{menasco1984closed, cromwell1993positive, feller2024homogeneous}.
Of course, not all relators in the group of a link are necessarily witnessed by plain spheres in a given diagram $D$. This is readily seen to be the case, even when $D$ is a diagram of the unknot. Since any finite set of meridians of a link can be realized as Wirtinger meridians in some diagram, one can regard the meridional rank conjecture as positing equality between the bridge number and a certain elusive sphere number of links, allowing immersed spheres which more easily evade the eye.

\newpage
\section{Figures}\label{sec:poem}

\noindent {\it Caught in tumbleweed, in perpetual slumber}\\
\noindent {\it They roll: round, squished, immersed, cucumber}\\
\noindent {\it Elusive, blistering, fierce}\\
\noindent {\it Faint! Diabolical! Spheres!}\\
\noindent {\it Does each breed breed the bridge number?}\\

\section*{Acknowledgements}
RB is partially supported by NSF grant DMS 2424734. AK and EP are partially supported by NSF grant DMS 2204349. We are particularly grateful to Nathan Dunfield for implementing a computation of the plain sphere number of a diagram in SnapPy and comparing the performance of the plain sphere and Wirtinger numbers on minimal diagrams through 16 crossings.

\bibliographystyle{plain}

\end{document}